\documentclass[10pt,a4paper,oneside,english,reqno]{amsart}
\usepackage[latin9]{inputenc}
\usepackage{babel}
\usepackage{amstext}
\usepackage{amsthm}
\usepackage{amssymb}
\usepackage[unicode=true,pdfusetitle,
 bookmarks=true,bookmarksnumbered=false,bookmarksopen=false,
 breaklinks=false,pdfborder={0 0 1},backref=false,colorlinks=false]
 {hyperref}

\makeatletter

\pdfpageheight\paperheight
\pdfpagewidth\paperwidth

\numberwithin{equation}{section}
\numberwithin{figure}{section}
  \theoremstyle{definition}
  \newtheorem{defn}{\protect\definitionname}
\theoremstyle{plain}
\newtheorem{thm}{\protect\theoremname}
  \theoremstyle{remark}
  \newtheorem{rem}{\protect\remarkname}
  \theoremstyle{plain}
  \newtheorem{lem}{\protect\lemmaname}

\@ifundefined{date}{}{\date{}}

\usepackage{amsmath}
\usepackage{graphics} 
\usepackage{epsfig} 
\usepackage{graphicx} 
\usepackage{epstopdf}

\hypersetup{urlcolor=blue, citecolor=red}

\makeatother

  \providecommand{\definitionname}{Definition}
  \providecommand{\lemmaname}{Lemma}
  \providecommand{\remarkname}{Remark}
\providecommand{\theoremname}{Theorem}

\begin{document}

\title[ADDENDUM TO: ``DACOROGNA-MOSER THEOREM'']{ADDENDUM TO: ``DACOROGNA-MOSER THEOREM ON THE JACOBIAN DETERMINANT
EQUATION WITH CONTROL OF SUPPORT''}

\email{pteixeira.ir@gmail.com}

\maketitle
\centerline{\scshape Pedro Teixeira}
\medskip
{\footnotesize
\centerline{Centro de Matemática da Universidade do Porto}
\centerline{Rua do Campo Alegre, 687, 4169-007 Porto, Portugal}
}

\renewcommand{\thefootnote}{}
\renewcommand{\thefootnote}{\arabic{footnote}} \setcounter{footnote}{0}
\begin{abstract}
In Dacorogna-Moser theorem on the pullback equation between two prescribed
volume forms, control of support of the solutions can be obtained
from that of the initial data, while keeping optimal regularity. Briefly,
given any two (nondegenerate) volume forms $f,\,g\in C^{r,\alpha}(\overline{\varOmega})$,
$r\geq0$ an integer and $0<\alpha<1$, with the same total volume
on a bounded connected open set $\varOmega\subset\mathbb{R}^{n}$
and such that $\text{supp}(f-g)\subset\varOmega$, there exists a
$C^{r+1,\alpha}$ diffeomorphism $\varphi$ of $\overline{\varOmega}$
onto itself such that
\[
\begin{cases}
\varphi^{*}(g)=f\\
\text{supp}(\varphi-\text{id})\subset\varOmega
\end{cases}
\]
This result answers a problem implicitly raised on page 14 of Dacorogna-Moser's
original article (``On a partial differential equation involving
the Jacobian determinant'', \emph{Ann. Inst. H. Poincaré Anal. Non
Linéaire }\textbf{7} (1990), 1\textendash 26), and fully generalizes
the solution to the particular case of $g\equiv1$ (prescribed Jacobian
PDE, $\text{det}\,\nabla\varphi=f$) given in the author's paper ``Dacorogna-Moser
theorem on the Jacobian determinant equation with control of support'',
\emph{Discrete Cont. Dyn. Syst. }\textbf{37} (2017), 4071-4089.
\end{abstract}

\section{\textbf{Introduction}}

This short note presents a simple solution to the problem of finding,
in the Hölder setting, optimal regularity solutions to the pullback
equation $\varphi^{*}(g)=f$ between two prescribed (nondegenerate)
volume forms $f,\,g$ which coincide near $\partial\varOmega$ and
with the same total volume over a bounded connected open set $\varOmega\subset\mathbb{R}^{n}$,
the solution being a diffeomorphism of $\overline{\varOmega}$ onto
itself that coincides with the identity near $\partial\varOmega$.
This problem has been implicitly raised in Dacorogna and Moser \cite[p.14, iv]{DM}
and again in \cite[p.18-19]{CDK}. In \cite{TE} the solution was
given for the particular case of $g\equiv1$ all over $\overline{\varOmega}$
(prescribed Jacobian PDE, $\text{det}\,\nabla\varphi=f$), but at
the time we could not devise the solution to the problem in its full
generality. It turns out that the deduction of the general solution
from the particular case mentioned above is in fact quite simple,
the paper \cite{TE} already containing all the tools needed for that
purpose. Here we complete the solution to the problem in question.

\section{\textbf{Optimal regularity pullback between prescribed volume forms
with control of support}}

We recall two definitions introduced in \cite{TE} for brevity of
expression, and which actually do \emph{not} coincide with the more
standard definitions of domain and collar.
\begin{defn}
(\emph{Domain}). A bounded, connected open set $\varOmega\subset\mathbb{R}^{n}$
is called here a \emph{domain}. Domains with $C^{r}$ boundary (briefly
$C^{r}$ domains), $r\geq1$ an integer, are defined in the usual
way \cite[p.338]{CDK}. A domain is smooth if it is $C^{\infty}$.
\end{defn}
\begin{defn}
(\emph{Collar of} $\overline{\varOmega}$). If $\varOmega\subset\mathbb{R}^{n}$
is a smooth domain, there is a smooth embedding $\zeta:\partial\varOmega\times[0,\infty)\hookrightarrow\overline{\varOmega}$
such that $\zeta(x,0)=x$ (collar embedding). For each $\epsilon>0$
we call $U_{\epsilon}:=\zeta(\partial\varOmega\times[0,\epsilon${]})
a (compact)\emph{ collar of} $\overline{\varOmega}$. \medskip{}

\noindent \textbf{Convention. }As usual, volume forms on domains are
identified with scalar functions, via the natural correspondence $f=f(x)dx^{1}\wedge\ldots\wedge dx^{n}$.\medskip{}

A simple formulation of the main result is the following (c.f. Theorem
2 below):
\end{defn}
\begin{thm}
\emph{(Dacorogna-Moser theorem - Case $\text{supp}(f-g)\subset\varOmega)$.
}Let $\varOmega\subset\mathbb{R}^{n}$ be a bounded connected open
set, $r\geq0$ an integer and $0<\alpha<1$. Given $f,\,g\in C^{r,\alpha}(\overline{\varOmega})$
with $f\cdot g>0$ in $\overline{\varOmega}$ satisfying:
\[
\begin{cases}
\int_{\varOmega}f=\int_{\varOmega}g\\
\mathcal{\text{\emph{supp}}}(f-g)\subset\varOmega
\end{cases}
\]
there exists $\varphi\in\text{\emph{Diff}}{}^{r+1,\alpha}(\overline{\varOmega},\overline{\varOmega})$
satisfying:
\[
\begin{cases}
\varphi^{*}(g)=f\\
\text{\emph{supp}}(\varphi-\text{\emph{id}})\subset\varOmega.
\end{cases}
\]
\end{thm}
\begin{proof}
Fix a smooth domain $\varOmega'$ such that $\overline{\varOmega'}\subset\varOmega$
and $\text{supp}(f-g)\subset\varOmega'$ (use e.g. \cite[Lemma 1]{TE}).
Then 
\[
\int_{\varOmega'}f=\int_{\varOmega}f-\int_{\varOmega\setminus\varOmega'}f=\int_{\varOmega}g-\int_{\varOmega\setminus\varOmega'}g=\int_{\varOmega'}g
\]
Fix a collar $U$ of $\overline{\varOmega'}$ small enough so that
$\text{supp}(f-g)$ is contained in the open set $\overline{\varOmega'}\setminus U$.
Clearly $f=g$ in a small neighbourhood of $U$ in $\overline{\varOmega'}$.
Apply Lemma 1 below to $f|_{\overline{\varOmega'}}$, $g|_{\overline{\varOmega'}}$,
$U$ and $\overline{\varOmega'}$, thus finding $\varphi\in\text{Diff}{}^{r+1,\alpha}(\overline{\varOmega},\overline{\varOmega})$
such that $\varphi^{*}(g|_{\overline{\varOmega'}})=f|_{\overline{\varOmega'}}$
and $\text{supp}(\varphi-\text{id})\subset\varOmega'$. Now, $\varphi$
extends by $\text{id}$ to the whole $\overline{\varOmega}$, in the
$C^{r+1,\alpha}$ class, thus providing the desired diffeomorphism
(as $\text{supp}(f-g)\subset\varOmega'$).
\end{proof}
\begin{rem}
As it is easily seen, in Theorem 1 the hypothesis $f,\,g\in C^{r,\alpha}(\overline{\varOmega})$
can be substantially weakened without compromising the conclusions.
It is enough to assume that (a) the restrictions of $f,\,g$ to $\varOmega$
are $C^{r,\alpha}$ and (b) $\int_{\varOmega}f=\int_{\varOmega}g<\infty$.
Actually, the hypothesis can be further weakened, for instance, the
conclusions still hold if $f$ and $g$ are arbitrary functions on
$\overline{\varOmega}$ such that (a') the restriction of $f$ and
$g$ to the closure of some smooth (sub)domain $\varOmega'$ satisfying
$\overline{\varOmega'}\subset\varOmega$ and $\text{supp}(f-g)\subset\varOmega'$
are in $C^{r,\alpha}$ and satisfy $f\cdot g>0$ (in $\overline{\varOmega'}$),
and (b') $\int_{\varOmega'}f=\int_{\varOmega'}g$ (note that the hypothesis
$\text{supp}(f-g)\subset\varOmega$ guarantees the existence of such
$\varOmega'$, see e.g. \cite[Lemma 1]{TE}). For then, the diffeomorphism
of $\overline{\varOmega'}$ onto itself solving $\varphi^{*}(g)=f$
in $\overline{\varOmega'}$ (obtained through Lemma 1) extends by
the identity to the whole $\overline{\varOmega}$ (in the $C^{r+1,\alpha}$
class) and $f=g$ in $\overline{\varOmega}\setminus\varOmega'$ by
hypothesis, thus everything agrees. 
\end{rem}
\begin{lem}
Let $\varOmega\subset\mathbb{R}^{n}$ be a bounded connected open
smooth set and $U$ a collar of $\overline{\varOmega}$. Let $r\geq0$
be an integer and $0<\alpha<1$. Given $f,\,g\in C^{r,\alpha}(\overline{\varOmega})$
with $f\cdot g>0$ in $\overline{\varOmega}$ satisfying:
\[
\begin{cases}
\int_{\varOmega}f=\int_{\varOmega}g\\
f=g & \text{in a neighbourhood of \,}U
\end{cases}
\]
there exists $\varphi\in\text{\emph{Diff}}{}^{r+1,\alpha}(\overline{\varOmega},\overline{\varOmega})$
satisfying:
\[
\begin{cases}
\varphi^{*}(g)=f\\
\varphi=\text{\emph{id}} & \text{in a neighbourhood of \,}U.
\end{cases}
\]
\end{lem}
\begin{proof}
(Lemma 1). Replace $f,\,g$ respectively by the positive volume forms
$|\lambda f|,\,|\lambda g|$ where $\lambda=\text{meas}\,\varOmega/\int_{\varOmega}\,f$,
thus getting $\int_{\varOmega}\,f=\text{meas}\,\varOmega=\int_{\varOmega}\,g$
and keeping the coincidence of $f$ and $g$ in a neighbourhood of
$U$. Apply in first place \cite[Theorem 1']{DM} to get a solution
$\varPhi\in\text{Diff}{}^{r+1,\alpha}(\overline{\varOmega},\overline{\varOmega})$
to $\text{det}\,\nabla\varPhi=f$ and then apply Lemma 2 below to
the data $\varPhi$ and $g$ to get a diffeomorphism $\phi$ satisfying
the conclusions of that lemma. The desired pullback between the two
original volume forms is then given by the diffeomorphism $\varphi=\phi^{-1}\circ\varPhi$.
\end{proof}
The next result generalizes \cite[Theorem 7]{TE} to the case of the
pullback equation between any two (nondegenerate) volume forms with
the same total volume (see the introduction to Section 7 in \cite{TE}
for the motivation). The proof, relying on Lemma 1, directly follows
the pattern of that of \cite[Theorem 7]{TE} and presents no difficulty.
The statement employs the convention $d(\emptyset,\partial\varOmega):=\text{inradius}\,\varOmega$,
introduced and explained in \cite[Section 7]{TE} (otherwise, for
any two nonvoid subsets of $\mathbb{R}^{n}$, $d(\cdot,*)$ denotes
the euclidean distance between them). 
\begin{thm}
Let $\varOmega\subset\mathbb{R}^{n}$ be a bounded connected open
set, $r\geq0$ an integer and $0<\alpha<1$. For each $0<c\leq R:=\text{\emph{inradius}}\,\varOmega$
there exists a neighbourhood $V_{c}$ of $\partial\varOmega$ in $\overline{\varOmega}$
such that: given any $f,\,g\in C^{r,\alpha}(\overline{\varOmega})$
with $f\cdot g>0$ in $\overline{\varOmega}$ and any $0<d\leq R$
satisfying:
\[
\begin{cases}
\int_{\varOmega}f=\int_{\varOmega}g\\
d(\text{\emph{supp}}(f-g),\,\partial\varOmega)\geq d
\end{cases}
\]
there exists $\varphi\in\text{\emph{Diff}}{}^{r+1,\alpha}(\overline{\varOmega},\overline{\varOmega})$
satisfying:
\[
\left\{ \begin{array}{llll}
\varphi^{*}(g)=f\\
\varphi=\text{\emph{id}} & in\:V_{d}.
\end{array}\right.
\]
\end{thm}

\subsection{Concordant solutions to the Jacobian determinant equation for volume
forms agreeing in a collar}
\begin{lem}
\emph{(Existence of concordant solutions).} Let $\varOmega\subset\mathbb{R}^{n}$
be a bounded connected open smooth set and $U$ a collar of $\overline{\varOmega}$.
Let $r\geq0$ be an integer and $0<\alpha<1$. Suppose that $\varPhi\in\text{\emph{Diff}}{}^{r+1,\alpha}(\overline{\varOmega},\overline{\varOmega})$
and $g\in C^{r,\alpha}(\overline{\varOmega})$, $g>0$ in $\overline{\varOmega}$,
satisfy:
\[
\begin{cases}
\int_{\varOmega}g=\text{\emph{meas}}\,\varOmega\\
g=\text{\emph{det}}\,\nabla\varPhi & \text{in a neighbourhood of }U.
\end{cases}
\]
Then, there exists $\phi\in\text{\emph{Diff}}{}^{r+1,\alpha}(\overline{\varOmega},\overline{\varOmega})$
satisfying:
\[
\begin{cases}
\text{\emph{det}}\,\nabla\phi=g\\
\phi=\varPhi & \text{in a neighbourhood of }U.
\end{cases}
\]
\end{lem}
\begin{proof}
Suppose that $\varPhi$, $g$ and $U$ are as in the statement. The
open set $\varOmega'=\overline{\varOmega}\setminus\varPhi(U)$ is
homeomorphic to $\varOmega$ and thus it is a domain containing $\varPhi(\text{supp}(g-\text{det}\,\nabla\varPhi))$.
Let 
\[
h=(g\circ\varPhi^{-1})\text{det}\,\nabla\varPhi^{-1}\in C^{r,\alpha}(\overline{\varOmega})
\]
Then 
\[
\int_{\varOmega}h=\int_{\varOmega}g=\text{meas}\,\varOmega
\]
 by the change of variables formula, and $h=1$ in a neighbourhood
of $\varPhi(U)$, thus 
\[
\int_{\varOmega'}h=\text{meas}\,\varOmega'\text{,}\quad\text{supp}(h-1)\subset\varOmega'\quad\text{and}\quad h>0.
\]
Apply \cite[Theorem 1]{TE} to the data $\varOmega'$ and $h|_{\overline{\varOmega'}}$,
getting $\varTheta\in\text{Diff}{}^{r+1,\alpha}(\overline{\varOmega'},\overline{\varOmega'})$
satisfying
\[
\begin{cases}
\text{det}\,\nabla\varTheta=h|_{\overline{\varOmega'}}\\
\text{supp}(\varTheta-\text{id})\subset\varOmega'
\end{cases}
\]
Extend $\varTheta$ by $\text{id}$ to the whole $\overline{\varOmega}.$
Then 
\[
\phi=\varTheta\circ\varPhi
\]
is the desired diffeomorphism, since
\begin{itemize}
\item $\phi=\text{id}\circ\varPhi=\varPhi$ ~in a neighbourhood of $U$
\item $\text{det}\,\nabla\phi=g$ ~by definition of $h$.
\end{itemize}
\end{proof}

\end{document}